\documentclass[11pt]{amsart}

\usepackage[all,cmtip]{xy}
\usepackage{amssymb}
\usepackage[english]{babel}
\usepackage{mathrsfs}

\newtheorem{theorem}{Theorem}[section]
\newtheorem{lemma}[theorem]{Lemma}

\newtheorem{corollary}[theorem]{Corollary}

\theoremstyle{definition}
\newtheorem{definition}[theorem]{Definition}

\newtheorem{remark}[theorem]{Remark}

\newtheorem{question}[theorem]{Question}

\newtheorem{claim}{Claim}

\counterwithin{subcase}{case}

\def\N{\mathbb{N}}

\def\Z{\mathbb{Z}}
\def\R{\mathbb{R}}

\def\grp#1{\langle{#1}\rangle}
\def\ssgp{$\mathrm{SSGP}$}

\def\Z{\mathbb{Z}}
\def\Q{\mathbb{Q}}

\def\ssgp{\mathrm{SSGP}}

\def\assgp{\mathrm{ASSGP}}
\def\nss{\mathrm{NSS}}

\def\minap{\mathrm{MinAP}}

\usepackage{color}
\newenvironment{revvy}{\color{magenta}}{}
\newenvironment{revdw}{\color{blue}}{}
\newenvironment{revsg}{\color{red}}{}

\def\by{\begin{revvy}}
\def\ey{\end{revvy}}
\def\bd{\begin{revdw}}
\def\ed{\end{revdw}}
\def\bg{\begin{revsg}}
\def\eg{\end{revsg}}

\def \Zet[#1]{\lceil #1 \rceil}

\def\e#1{e}

\def\U#1{\mathbb{U}_{\leq #1}}

\DeclareMathOperator{\Iso}{Iso}

\begin{document}
\title[Generic dense free NSS subgroups]{Generic dense free subgroups of the isometry group of the Urysohn space are NSS }

\author[V.H Ya\~nez]{V\'{\i}ctor Hugo Ya\~nez}
\address{School of Mathematical Sciences and LPMC, Nankai University, Tianjin 300071, P.R. China}
\email{vhyanez@nankai.edu.cn}
\thanks{The author acknowledges the partial support of their research by the Fundamental Research Funds for the Central Universities, the National Natural Science Foundation of China (NSFC) grant 12271263, and the China Postdoctoral Science Foundation Tianjin Joint Support Program under Grant Number 2025T001TJ}

\begin{abstract}
The isometry group of the bounded Urysohn space, $G = \mathrm{Iso}(\U{1})$ is a central object in the study of Polish groups and topological dynamics. It is known that generic sequences in $G$ generate algebraically free dense subgroups. In this paper, we show that such generic free subgroups exhibit strong geometric rigidity. Specifically, we prove that for a comeager set of sequences generating dense free subgroups $F\leq G$, every non-trivial element $h\in F$ acts with maximal metric displacement, satisfying $\sup_{n\in \N} d(h^n(x),x) = 1$ for every $x \in \U{1}$. 
As a consequence, these generic subgroups satisfy the \emph{no small subgroup} ($\nss$) property. We note that the method naturally extends to the full isometry group $\mathrm{Iso}(\mathbb{U})$ of the classical Urysohn space.
\end{abstract}

\subjclass[2010]{Primary: 22A05; Secondary: 20E05,54E50,54H11}

\keywords{Urysohn space, isometry group, generic free subgroup, no small subgroups, small subgroup generating property, Polish group}

\maketitle

\section{Introduction}

The topological spaces appearing in this paper shall be assumed to be Hausdorff. A topological group $G$ is said to have \emph{no small subgroup ($\nss$)} if some neighbourhood of the identity contains no non-trivial subgroup. The $\nss$ property plays a central role in the solution of Hilbert's Fifth Problem: a locally compact group is Lie if and only if it is $\nss$. 

An opposing property is the following: a topological group $G$ is said to have the \emph{small subgroup generating property} ($\ssgp)$ if each open neighbourhood of the identity of $G$ contains a family $\mathcal{H}$ of subgroups of $G$ such that the subgroup $\grp{\bigcup \mathcal{H}}$ is dense in $G$.  If density is replaced by equality (e.g, $\grp{\bigcup \mathcal{H}} = G$ always holds), then $G$ satisfies the (stronger) \emph{algebraic} small subgroup generating property ($\assgp$). Each $\assgp$ group is naturally $\ssgp$, and an $\ssgp$ group admits no non-trivial continuous homomorphism to an $\nss$ group (see \cite{YVH_Hilbert}).

The $\ssgp$ property was introduced by Gould \cite{CG, Gould,Gould2} as a method for constructing minimally almost periodic $\minap$ groups. Since extreme amenability implies $\minap$, many large Polish groups provide natural examples of $\minap$ groups. In particular, Pestov proved that $\Iso(\mathbb{U})$ is extremely amenable. The properties are related as follows:
\begin{equation}
\assgp \rightarrow \ssgp \rightarrow \minap \leftarrow \text{extremely amenable}.
\end{equation}

Free groups are known to admit both $\nss$ and $\ssgp$ group topologies. For instance, a classical theorem of Sipacheva and Uspenskij \cite{SipachevaUspenskii}  shows that a free topological group $F(X)$ is $\nss$ if and only if the base space $X$ admits a continuous metric. On the other hand, a result of Shakhmatov and the author is the following: every non-commutative free group $F$ of infinite rank admits an $\assgp$ group topology \cite{SY_FreeSSGP}. Furthermore, the topology is metrizable whenever $F$ has exactly \emph{countably infinite} rank.

The bounded and unbounded Urysohn isometry groups provide canonical examples of universal Polish groups with rich generic structure \cite{Uspenskij1990}. A theorem of Kechris and Rosendal \cite{KR07} implies that generic sequences in these groups generate algebraically free dense subgroups. Since $\minap$ passes to dense subgroups, these free groups inherit $\minap$ subspace topologies.

This leads naturally to the following question:

\begin{question}
Are these generic free groups $\ssgp$?
\end{question}

We answer the question in the negative by proving that generic algebraically free groups exhibit strong $\nss$ behavior. The main result of this paper is the following: 

\begin{theorem} \label{this:Polish:ssgp:thm}
Let $G = \mathrm{Iso}(\U{1})$ equipped with the topology of pointwise convergence. For each $x \in \U{1}$, the set of sequences $\bar{g} = (g_i)_{i \in \N} \in G^\N$ such that every non-trivial element $h$ in the generated subgroup $F_{\bar{g}} = \grp{g_i: i \in \N}$ satisfies
\[
\sup_{n \in \Z} d(h^n(x),x) = 1
\]
is comeager in $G^\N$.
\end{theorem}

While it is a known from generic dynamics that \emph{primitive} words act with maximal displacement, Theorem \ref{this:Polish:ssgp:thm} establishes the same conclusion simultaneously for every non-trivial word. Its proof is reserved for Section \ref{this:Polish:ssgp:thm:proof:section}.

A result of Kechris and Rosendal \cite{KR07} establishes that the set of sequences in $\mathrm{Iso}(\U{1})^\N$ generating algebraically free and topologically dense subgroups is comeager. Since the intersection of two comeager sets remains comeager, Theorem \ref{this:Polish:ssgp:thm} immediately yields the following:

\begin{corollary} \label{this:ssgp:corollary}
Comeagerly many sequences in $\mathrm{Iso}(\U{1})^\N$ generate algebraically free, dense subgroups whose non-trivial elements satisfy \[
\sup_{n \in \Z} d(h^n(x),x) = 1.
\] Consequently, the set of sequences generating $\ssgp$ subgroups is meager in $\mathrm{Iso}(\U{1})^\N$.
\end{corollary}

The proof of the corollary is also given in Section \ref{this:Polish:ssgp:thm:proof:section}.

\section{Background and notation}

Given a non-empty set $Z$, we denote by $F(Z)$ the free group generated by $Z$. Any free group admits the structure of a metric space via the \emph{word metric}.

\begin{definition} \label{def:word_metric}
Let $F(Z)$ be the free group generated by a finite set $Z$, and let $Z^{\pm 1} = Z \cup Z^{-1}$ denote the symmetric alphabet of generators and their inverses. For any element $g \in F(Z)$, its \emph{freely reduced word length}, denoted $|g|$, is defined as the number of letters in the unique reduced word representing $g$ over $Z^{\pm 1}$.

The \emph{word metric} on $F(Z)$ is the distance function $d_Z: F(Z) \times F(Z) \to \N \cup \{0\}$ defined by
\begin{equation}
    d_Z(g, h) = |g^{-1}h| \quad \text{for all } g, h \in F(Z).
\end{equation}
\end{definition}

\begin{remark} \label{lem:left_invariance}
The word metric $d_Z$ is left-invariant. That is, for all $g, h, k \in F(Z)$, left-multiplication by $k$ acts as a bijective isometry on the metric space $(F(Z), d_Z)$, hence
$d_Z(kg, kh) = d_Z(g, h)$.
\end{remark} 

\begin{definition} \label{def:urysohn_iso}
We denote by $\U{1}$ the universal Urysohn space of diameter $\leq 1$. The group of isometries $G = \mathrm{Iso}(\U{1})$ is equipped with the \emph{topology of pointwise convergence}. A basic open neighbourhood of $h \in G$ is determined by a finite subset $A \subseteq \U{1}$ and a tolerance $\delta > 0$, defined as:
$$ V_{h, A, \delta} = \{ g \in G : d(g(a), h(a)) < \delta \text{ for all } a \in A \} $$
Under this topology, $G$ is a Polish group.
\end{definition}

Let $G$ be a Polish group. A property of countable subgroups of $G$
is said to be \emph{generic} if the set of sequences
$(g_1,g_2,\dots)\in G^\N$ generating subgroups with that property
forms a comeager subset of $G^\N$. A theorem of Kechris and Rosendal \cite{KR07} establishes a generic
freeness phenomenon for several large Polish groups, including
$\Iso(\U{1})$: the set of sequences generating algebraically free
dense subgroups is comeager in $\Iso(\U{1})^\N$.

To state this precisely, we use the following terminology:

\begin{definition} \label{def:this:replacement:map}
Let $G$ be a topological group and let $F(Z)$ be the free group over a countably infinite set $Z = \{z_i: i \in \N\}$. For each tuple $\bar{g} = (g_1,g_2,\dots ) \in G^\N$ we define the replacement homomorphism $\Psi_{\bar{g}}: F(Z) \to G$ by \begin{equation} \label{this:eval:map:def}
\Psi_{\bar g}(z_i)=g_i \qquad (i\in\N).
\end{equation}
A sequence $\bar{g} \in G^\N$ is said to generate an algebraically free subgroup
if the homomorphism $\Psi_{\bar g}$ is injective.
\end{definition}

\section{Geodesic realizations of reduced words} \label{baire:section}

\begin{lemma} \label{lem:suffix_space}
Let $Z$ be a finite set and $F(Z)$ be the free group equipped with the word metric $d_Z$. Let $v = l_1 l_2 \dots l_m$ be a reduced word of length $m$ over $Z^{\pm 1}$. For each $\eta > 0$ satisfying $m \eta < 1$, there exists a finite metric space $(B, d_B)$ with $\mathrm{diam}(B) < 1$, where the underlying set $B = \{ s_1, s_2, \dots, s_{m+1} \} \subseteq F(Z)$ 
satisfies:
\begin{enumerate}
    \item $d_B(s_1, s_{m+1}) = m\eta$.
    \item For each generator $z_i \in Z$, the restriction of the left-multiplication action to $B$, 
    \[ E_i = \{ (u, z_i u) \mid u \in B \text{ and } z_i u \in B \} \] 
    is a well-defined finite partial isometry on $B$.
    \item For any set $X$ containing $B$, if $\psi: F(Z) \to \mathrm{Sym}(X)$ is any group homomorphism such that the assigned bijection $\psi(z_i)$ extends the partial function $E_i$ for all $z_i \in Z$, then \[
    \psi(v)(s_{m+1}) = s_1.
    \]
    \item Let $x \not \in B$ and let $\delta  > 0$ satisfy $0 < \delta < 1 - m\eta$. Then $(B, d_B)$ admits a one-point extension $B' = B \cup \{x\}$ with diameter $\operatorname{diam}(B') < 1$ such that \begin{equation} \label{this:onepoint:condition}
    d_{B'}(b,x) = d_{B'}(x,b) = \delta + d_B(s_{m+1},b) \text{ for all } b \in B.
    \end{equation}
\end{enumerate}
\end{lemma}

\begin{proof}
Since $v=l_1\cdots l_m$ is reduced, the successive suffixes \[ s_r=l_r\cdots l_m \quad (1\le r\le m), \qquad s_{m+1}=e,
\] 
are pairwise distinct and form the vertices of the unique geodesic segment from $e$ to $v$ in the Cayley graph of $F(Z)$. The suffixes correspond to the successive states obtained while evaluating the word $v$ from right to left.
Let \[
B = \{ s_1, s_2, \dots, s_{m+1} \}.
\] 

We construct the metric $d_B$ for $B$ by equipping this finite set $B$ with the word metric $d_Z$, scaled by $\eta$:
$$ d_B(u, u') = \eta d_Z(u, u')  \quad \text{for all } u, u' \in B. $$
Since the suffixes form a geodesic segment in the Cayley graph, we have
\[
d_Z(s_i,s_j)=|i-j|
\]
for all \(1\le i,j\le m+1\). Hence
\[
d_B(s_i,s_j)=\eta|i-j|.
\]
Therefore, $\operatorname{diam}(B)=m\eta<1$ holds, and in particular
\[
d_B(s_1,s_{m+1})=m\eta.
\]

For each $z_i \in Z$, let $L_{z_i} : F(Z) \to F(Z)$ denote the left-multiplication bijection defined by $L_{z_i}(x) = z_i x$. 
The partial map $E_i$ is obtained by restricting $L_{z_i}$ to pairs of points lying in $B$:
\[
E_i = L_{z_i} \cap (B\times B).
\]
Viewed as a partial function on $B$, its domain is $\operatorname{dom}(E_i) = \{ u \in B \mid z_i u \in B \}$. Since $L_{z_i}$ is bijective, the subset $E_i$ of its ordered pairs is injective and a well-defined function on $\operatorname{dom}(E_i)$. By the left-invariance of the word metric (Remark \ref{lem:left_invariance}), $L_{z_i}$ is an isometry on $(F(Z),d_Z)$. We now note that for arbitrary $u, u' \in \operatorname{dom}(E_i)$:
$$ d_B(z_i u, z_i u') = \eta d_Z(z_i u, z_i u')  = \eta d_Z(u, u') = d_B(u, u') $$
Thus, $E_i$ is a finite partial isometry on the metric space $(B, d_B)$, satisfying (2).

To prove (3), let $X$ be a set containing $B$, and assume $\psi: F(Z) \to \mathrm{Sym}(X)$ is a group homomorphism such that $\psi(z_i)$ extends $E_i$ for all $z_i \in Z$. We must evaluate the composition:
$$ \psi(v)(s_{m+1}) = \left[ \psi(l_1) \circ \psi(l_2) \circ \dots \circ \psi(l_m) \right] (s_{m+1}) $$
To simplify calculations, for any $l \in Z^{\pm 1}$, we let $E_l^*$ denote the corresponding partial relation on $B$: if $l = z_i$, then $E_l^* = E_i$; if $l = z_i^{-1}$, then $E_l^* = E_i^{-1}$. Since $\psi$ extends each $E_i$, the bijection $\psi(l)$ extends $E_l^*$ for every $l\in Z^{\pm1}$. In addition, $\psi(l) \restriction_{\operatorname{dom}(E_l^*)}$ acts exactly as $E_l^*$, i.e., \textbf{left multiplication by $l$}.

\begin{claim} \label{this:reverse:path:claim}
Let $x_m = s_{m+1}$, and for each integer $0 \leq j \leq m-1$, let $x_j = \left[\psi(l_{j+1}) \circ \dots \circ \psi(l_m)\right] (s_{m+1})$. For all $0 \le j \le m$, it holds that $x_j = s_{j+1}$.
\end{claim}
\renewcommand{\qedsymbol}{$\blacksquare$}
\begin{proof}[Proof of Claim]
We proceed by finite reverse induction on $j$. 

\emph{Basis of induction:} By definition, $x_m =  s_{m+1}$. 

\emph{Inductive step:} Assume the claim holds for some $1\le j\le m$. We must show it holds for $j-1$. By definition of our sequence, we have $x_{j-1} = \psi(l_j)(x_j)$. 
By the inductive hypothesis, $x_j = s_{j+1}$, meaning we must evaluate $\psi(l_j)(s_{j+1})$.
By construction of the suffixes, $s_j = l_j s_{j+1}$. Since both $s_{j+1}$ and $s_j$ belong to $B$, the pair $(s_{j+1}, s_j)$ explicitly belongs to the relation $E_{l_j}^*$. 
Hence $s_{j+1}\in\operatorname{dom}(E_{l_j}^*)$.
Since $\psi(l_j)$ is an extension of $E_{l_j}^*$, it must map $s_{j+1}$ directly to $s_j$. 
Thus, $x_{j-1} = \psi(l_j)(s_{j+1}) = s_j$, which proves the claim for $j-1$. This completes the induction.
\end{proof}

Evaluating Claim \ref{this:reverse:path:claim} at $j=0$, the full composition evaluates to $x_0 = s_1$. Since $s_1 = l_1 \dots l_m = v$, we conclude $\psi(v)(s_{m+1}) = x_0 = s_1$, satisfying (3).

For item (4), let $x \not \in B$ and $\delta > 0$ satisfying $\delta < 1 - m\eta$ be arbitrary. Define the set $B' = B \cup \{x\}$ and let $d_{B'} : B' \times B' \to \R$ be the function which coincides with $d_{B}$ on $B \times B$, makes $d_{B'}(x,x) = 0$ hold, and satisfies \eqref{this:onepoint:condition}. By definition, $B'$ has diameter $\delta + m\eta$. It suffices to prove the following: \begin{claim}
The map $d_{B'}$ is a well-defined metric on $B'$.
\end{claim}
\begin{proof}
The function $d_{B'}$ is clearly symmetric and non-negative. Note that $d_{B'}(u,v) = 0$ holds if and only if $u = v$ holds; indeed: if $u,v \in B$ the equality follows from evaluating $d_B$. If one point is $x$ and $v \in B$ then $d_{B'}(x,v) \geq \delta > 0$ by \eqref{this:onepoint:condition}.

It then suffices to prove that the triangle inequality holds. Let $b,c \in B$ be arbitrary points. We note that \begin{align*}
d_{B'}(x,c) &= \delta + d_B(s_{m+1},c) \\
&\leq \delta + d_B(s_{m+1},b) + d_B(b,c) \\
&= d_{B'}(x,b) + d_{B'}(b,c)
\end{align*}
Similarly, \[
d_{B'}(b,c) \leq d_{B'}(b,x) + d_{B'}(x,c)
\]
follows from \[
d_{B}(b,c) \leq d_{B}(b,s_{m+1}) + d_{B}(s_{m+1},c).
\]
This proves that $d_{B'}$ is a well-defined metric on $B'$.
\end{proof}
From the previous claim, we are able to conclude the proof of item (4).
\end{proof}

\section{A lemma about escaping words in $\mathrm{Iso}(\U{1})$}

The following lemma is a standard bounded version of the free amalgamation
construction for metric spaces, as used for instance by Tent and Ziegler
\cite[Example 2.2]{TentZiegler}. We include the statement for convenience.

\begin{lemma} \label{lem:free_amalgamation}
Let $(A, d_A)$ and $(B, d_B)$ be finite metric spaces of diameter at most $1$ intersecting at a single common point $A \cap B = \{x\}$. Let $Y = A \cup B$, we define the cross-distances between the disjoint parts via the truncated maximal metric over $x$:
\begin{equation}
d_Y(a, b) = \min(1, d_A(a, x) + d_B(x, b)) \quad \text{for all } a \in A \text{ and } b \in B.
\end{equation}
Equipped with this extension and the native metrics $d_A$ and $d_B$, the space $(Y, d_Y)$ is a metric space of diameter at most $1$. Furthermore, because $\U{1}$ is the universal Urysohn space of diameter $1$, if $A$ is isometrically embedded in $\U{1}$, then $Y$ can be isometrically embedded into $\U{1}$ extending the native embeddings of $A$ and $B$.
\end{lemma}

\begin{lemma} \label{lem:word_escape}
Let $F(Z)$ be the free group over a finite set $Z = \{z_1, \dots, z_k\}$. Let $w \in F(Z) \setminus \{e\}$ be a non-trivial reduced word. Let $U \subseteq \mathrm{Iso}(\U{1})^k$ be a non-empty open set in the product topology of pointwise convergence. For any base point $x \in \U{1}$ and any $0 < \epsilon < 1$, there exists a tuple of isometries $H = (h_1, \dots, h_k) \in \mathrm{Iso}(\U{1})^k$ and an integer $N \ge 1$ such that the evaluation homomorphism $\psi_H: F(Z) \to \mathrm{Iso}(\U{1})$ (defined by $\psi_H(z_i) = h_i$) satisfies \[ 
d(\psi_H(w)^N(x), x) > \epsilon.
\]
\end{lemma}
\begin{proof}
Since $U$ is open, it contains a basic open set. Without loss of generality, assume $U = U_1 \times \dots \times U_k$, where each $U_i = V[f_i, A_i, \delta_i]$ is open in $\text{Iso}(\U{1})$. To ensure our arbitrary base point $x \in \U{1}$ is included, we define the master finite domain $A = \{x\} \cup \bigcup_{i=1}^k A_i$ and set $\delta = \min\{\delta_i : i =1,\dots,k\}$. Without loss of generality, we may select our open sets to be $U'_i = V[f_i, A, \delta] \subseteq U_i$, replacing $U_i$ with $U'_i$ so they share the common domain $A$ containing $x$. For each $i = 1, \dots, k$, we define the finite partial isometry $p_i = f_i|_A : A \to \U{1}$ with domain $A$.

Since $w$ is a non-trivial word, the length of the reduced word $w^N$ is arbitrarily large as $N$ tends to infinity. Select a large enough positive integer $N$ and a constant $\eta > 0$ such that the length $m = |w^N|$ satisfies $\epsilon < m \eta < 1$. We let $v = w^N$. 

We now select the finite metric space $(B, d_B)$ from Lemma \ref{lem:suffix_space} for the word $v$ and $\eta$. Following the lemma, we let $B = \{ s_1, s_2, \dots, s_{m+1} \}$ be the underlying set, and $\{E_i: 1 \leq i \leq k\}$ be the set of partial isometries on $B$. 

Since  $\epsilon < m \eta < 1$ holds, select $\rho >0$ satisfying $\rho < 1 - m \eta$. Consider the amalgamation space $B' = B \cup \{x\}$ constructed in item (iv) of Lemma \ref{lem:suffix_space} for this selection of $\rho$. Consider the  amalgamation space $Y = A \cup B'$ over the common point $x$ as given in Lemma \ref{lem:free_amalgamation}. Recall that the cross distances between $A$ and $B'$ are given by:
\begin{equation}
d_Y(a, b) = \min(1, d_A(a, x) + d_{B'}(x, b)) \quad \text{for all } a \in A \text{ and } b \in B'.
\end{equation}
Since $Y$ is a finite metric space of diameter at most $1$ and
$A \subseteq \U{1}$, the universality of $\U{1}$ yields an
isometric embedding of $Y$ into $\U{1}$ extending the identity
embedding of $A$. Henceforth, we identify $Y$ with its isometric image inside $\U{1}$. 

For each $i=1,\dots,k$, define a partial map
$p'_i = p_i \cup E_i$ on $Y$. Since $\mathrm{dom}(E_i) \subseteq B$, $\mathrm{dom}(p_i) \subseteq A$ and $A \cap B = \emptyset$ hold, the union $p'_i$ is a well-defined partial isometry on $Y$. By ultrahomogeneity of $\U{1}$ we may now select a family of isometries $h_1,\dots,h_k \in \text{Iso}(\U{1})$  such that $h_i$ is an extension of $p'_i$ for all $i=1,\dots,k$. Since $h_i$ coincides with $p_i = f_i \restriction_A$ on the set $A$, we now have \[
H = (h_1,\dots,h_k) \in U_1 \times \dots \times U_k = U. 
\]
Let us now take the (unique) homomorphism $\psi_H: F(Z) \to \text{Iso}(\U{1})$ which maps $\psi_H(z_i) = h_i$ on the set of generators $Z$. Since $\psi_H(z_i) = h_i$ is an extension of $E_i$ (by construction), we now apply item (3) of Lemma \ref{lem:suffix_space} to deduce \[
\psi_H(v)(s_{m+1}) = s_1.
\]
Since $d_Y(x,s_{m+1}) = \rho$ and $\psi_H(v)$ is an isometry, we deduce that \[
d(\psi_H(v)(x),s_1) = d(\psi_H(v)(x),\psi_H(v)(s_{m+1})) = d(x,s_{m+1}) = d_Y(x,s_{m+1}) = \rho.
\]
Furthermore, the reverse triangle inequality implies \[
d(\psi_H(v)(x),x) \geq  d(s_1,x) -  d(\psi_H(v)(x),s_1) = d_{B'}(s_1,x) - \rho.
\]
By item (4) of Lemma \ref{lem:suffix_space} know that $d_{B'}(s_1,x) = \rho + d_B(s_{m+1},s_1)$. Applying the above inequality we obtain \[
d(\psi_H(v)(x),x) \geq d_B(s_{m+1},s_1) = m\eta > \epsilon.
\]
as desired. Since $\psi_H$ is a homomorphism and $v = w^N$ it suffices to note that $d(\psi_H(v)(x),x) = d(\psi_H(w)^N(x),x)$ holds. This concludes our proof.
\end{proof}

\section{Proof of the main result} \label{this:Polish:ssgp:thm:proof:section}

\subsection{Proof of Theorem \ref{this:Polish:ssgp:thm}}

Fix arbitrary $x \in \U{1}$ and a rational $q \in (0,1) \cap \Q$. Consider the basic open set \[U_{x, q} = \{ g \in G : d(g(x), x) < q \}.\] 
Along with it, we define the closed set \[
S_{x, q} = \{ g \in G : d(g^n(x), x) \le q \text{ for all } n \in \N \}.
\] 
Let $F(Z)$ be a free group over a countably infinite set $Z = \{z_i: i \in \N\}$. Let $w \in F(Z)$ be a non-trivial reduced word. For each sequence $\bar{g} \in G^\N$ let $\Psi_{\bar{g}}: F(Z) \to G$ denote the replacement homomorphism satisfying \eqref{this:eval:map:def}. Write $w = l^{\varepsilon_1}_1 l^{\varepsilon_2}_2 \cdots l^{\varepsilon_m}_m$ in reduced form, for some $l_1,\dots,l_m \in Z$ and $\varepsilon_1,\dots,\varepsilon_m \in \{-1,1\}$. We now define the \emph{evaluation map} $\Phi_w: G^\N \to G$ of $w$ as 
\begin{equation} \label{this:bigPhi:eval} 
\Phi_w(\bar{g}) =\Psi_{\bar{g}}(w) = \Psi_{\bar{g}}(l_1)^{\varepsilon_1} \Psi_{\bar{g}}(l_2)^{\varepsilon_2} \cdots \Psi_{\bar{g}}(l_m)^{\varepsilon_m} \in \grp{\{g_i: i \in \N\}}.
\end{equation}

Since multiplication and inversion are continuous in $G$,
the evaluation map $\Phi_w$ is continuous, and the pre-image $\Phi_w^{-1}(S_{x, q})$ of $S_{x,q}$ under $\Phi_w$ is closed in $G^\mathbb{N}$.
\begin{claim} \label{this:empty:interior:claim}
The set $\mathcal{M}_{x,q,w} = \Phi_w^{-1}(S_{x, q})$ has empty interior.
\end{claim}
\begin{proof}
By contradiction, assume $\mathcal{M}_{x,q,w}$ contains a non-empty basic open set $U \subseteq  G^\mathbb{N}$. Then, $U$ is supported on a finite set of coordinates $I \subseteq \N$. Let
\[
J=\{j\in\N : z_j \text{ appears in } w\}.
\]
We then take the set $K = I \cup J$ and select the subset of letters $Z_K = \{z_k : k \in K\} \subseteq Z$. Let $W \subseteq G^K$ be the projection of $U$ onto the $K$-coordinates. By construction we have \[
W \times G^{\mathbb{N} \setminus K} \subseteq U.
\]

For any finite tuple $H \in W$, we let $\psi_H: F(Z_K) \to G$ be the finite 
replacement homomorphism 
which maps $z_k \mapsto h_k$ for each $k \in K$. Let $\bar{g} \in U$ be any infinite sequence whose projection onto the $K$-coordinates is exactly $H$ (i.e., any sequence in $H \times G^{\mathbb{N} \setminus K}$). Since $w$ only uses letters in $Z_K$, and $\bar{g}$ agrees with $H$ on the coordinates indexed by $K$ we have:
\[
\Phi_w(\bar g)=\Psi_{\bar g}(w)=\psi_H(w)\in S_{x,q}.
\]
Since $\bar{g} \in H \times G^{\mathbb{N} \setminus K}$ was arbitrary, we have proven the following: for any tuple $H \in W$ the inclusion $\psi_H(w) \in S_{x,q}$ holds. We now apply Lemma \ref{lem:word_escape} to select a finite tuple $H \in W$ and an integer $N$ such that $d(\psi_H(w)^N(x), x) > q$. Finally, select an arbitrary element $\bar{g}$ such that $\bar{g} |_K = H$ to obtain an element $\bar{g} \in U$ which satisfies \[
d(\Phi_w(\bar{g})^N(x), x) > q.
\]
This implies that $\Phi_w(\bar{g}) \not \in S_{x,q}$ and hence $\bar{g} \not \in \mathcal{M}_{x,q,w}$, which is a contradiction. We conclude that $\Phi_w^{-1}(S_{x, q})$ contains no open sets.
\end{proof}
By Claim \ref{this:empty:interior:claim}, for every non-trivial reduced word $w \in F(Z) \setminus \{e\}$, the set $\mathcal{M}_{x,q,w}$ is closed and nowhere dense in $G^\mathbb{N}$. Since $Z$ is countable, the free group $F(Z)$ is countable, and thus the union 
\[ \mathcal{M}_x = \bigcup_{q \in \Q \cap (0,1)} \bigcup_{w \neq e} \mathcal{M}_{x,q,w} \] 
is a meager subset of $G^\N$. Hence
\[
\mathcal C_x = G^\N\setminus \mathcal M_x
\]
is comeager in $G^\N$.

\begin{claim}
Let $\bar{g} \in \mathcal{C}_x$. If $h \in F_{\bar{g}} = \grp{g_i: i \in \N}$ is arbitrary and non-trivial, then \[
\sup_{n \in \N} d(h^n(x),x) = 1.
\]
\end{claim}
\begin{proof}
Let $h \in  F_{\bar{g}} \setminus \{e\}$ be arbitrary. Then $h = \Phi_w(\bar{g})$ for some non-trivial word $w$. Fix a rational $q \in \Q \cap (0,1)$. Since $\bar{g} \not \in \mathcal{M}_{x,q,w}$, we have $ \Phi_w(\bar{g}) \not \in S_{x,q}$. This implies that there exists $n \in \Z$ such that \[
d(h^n(x),x) > q.
\]
Since the metric on $\U{1}$ is bounded by $1$, the previous inequality for any choice of $q < 1$ implies that $\sup_{n \in \Z} d(h^n(x),x) = 1$.
\end{proof}
To conclude our proof, we let $\mathcal{F} \subseteq G^{\N}$ be the set of sequences which generate dense and algebraically free subgroups of $G$. The intersection $\mathcal{G}_x = \mathcal{F} \cap \mathcal{C}_x$ is then comeager in $G^\N$, and its sequences satisfy items (a) and (b), as desired.   
\qed

\subsection{Proof of Corollary \ref{this:ssgp:corollary}}

Let $x \in \U{1}$ be arbitrary, and let $\bar{g} \in \mathcal{G}_x = \mathcal{F} \cap \mathcal{C}_x$ be an arbitrary sequence. The subgroup $F_{\bar{g}}$ of $\mathrm{Iso}(\U{1})$ is then algebraically free and dense in $\mathrm{Iso}(\U{1})$. Let $q = 1/2$ and consider the following open neighbourhood of the identity of $\mathrm{Iso}(\U{1})$: \[U_{x, q} = \{ g \in G : d(g(x), x) < q \}.\] 
It suffices to prove the following: \begin{claim}
The intersection $U_{x, q} \cap F_{\bar{g}}$ contains no non-trivial subgroup of $F_{\bar{g}}$.
\end{claim}
\begin{proof}
By contradiction, assume there exists $h \in F_{\bar{g}} \setminus \{e\}$ such that $\grp{h} \subseteq U_{x, q}$. By item (b) of Theorem \ref{this:Polish:ssgp:thm}, there exists $n \in \Z$ such that $d(h^n(x),x) > q$. This implies that $h^n \not \in U_{x,q}$, which contradicts the selection of $h$. This proves that $U_{x, q}$ contains no non-trivial subgroups of $F_{\bar{g}}$.
\end{proof}
Since $F_{\bar g}$ carries the subspace topology inherited from
$\mathrm{Iso}(\U{1})$, the neighbourhood
$U_{x,q}\cap F_{\bar g}$ is an open neighbourhood of the identity in
$F_{\bar g}$ containing no non-trivial subgroup. We deduce that $F_{\bar{g}}$ is $\nss$ under the subspace topology. Hence $F_{\bar g}$ is $\nss$, and therefore cannot be $\ssgp$.
\qed

\subsection{The unbounded Urysohn space}

The arguments above extend directly to the full Urysohn space $\mathbb{U}$.

\begin{remark}
The boundedness assumption on $\U{1}$ was used only in two places:
\begin{enumerate}
    \item to ensure that the finite metric spaces constructed in Lemma \ref{lem:suffix_space} have diameter strictly less than $1$; and
    \item in Lemma \ref{lem:free_amalgamation}, where cross-distances were truncated by the operation
    \[
    d(a,b)=\min(1,d(a,x)+d(x,b)).
    \]
\end{enumerate}

For the classical Urysohn space $\mathbb{U}$, the truncation is unnecessary.
Indeed, if $(A,d_A)$ and $(B,d_B)$ are finite metric spaces intersecting
in a single point $x$, one may define
\[
d(a,b)=d_A(a,x)+d_B(x,b)
\]
for $a\in A$ and $b\in B$.
The resulting amalgam is again a metric space, and universality and ultrahomogeneity of $\mathbb{U}$ allow the remainder of the proof to proceed verbatim. Consequently, the analogues of Theorem
\ref{this:Polish:ssgp:thm} and Corollary
\ref{this:ssgp:corollary} hold for $\Iso(\mathbb{U})$, with
\[
\sup_{n\in\mathbb \N} d(h^n(x),x)=\infty
\]
replacing the condition
\[
\sup_{n\in\mathbb \N} d(h^n(x),x)=1.
\]
\end{remark}


\begin{thebibliography}{99}          

\bibitem{CG} W. Comfort and F. R. Gould, {\it Some classes of minimally almost periodic topological groups}, Appl. Gen. Topol.16 (2015), 141--165.

\bibitem{Gould} F. Gould, {\it On certain classes of minimally almost periodic groups}, Thesis (Ph.D.), Wesleyan University (2009),
136 pp. ISBN: 978-1109-22005-6.

\bibitem{Gould2} F. Gould, {\it An SSGP topology for $\Z^\omega$}, Topology Proc. 44 (2014), 389--392.

\bibitem{KR07} A. S. Kechris and C. Rosendal, \emph{Turbulence, amalgamation, and generic automorphisms of homogeneous structures}, Proc. Lond. Math. Soc. (3) 94 (2007), no. 2, 302--350.

\bibitem{SY_FreeSSGP}  D. Shakhmatov and V.H. Ya\~nez, {\it SSGP topologies on free groups of infinite rank}, Topol. Appl. 259 (2019), 384--410.

\bibitem{SipachevaUspenskii} O.V. Sipacheva and V.V. Uspenskij, {\it Free topological groups with no small subgroups, and Graev metrics}, Vestnik Moskov. Univ. Ser. 1. Mat. Mekh., 1987, no. 4, 21--24.

\bibitem{TentZiegler} K. Tent and M. Ziegler, {\it On the isometry group of the Urysohn space}, J. Lond. Math. Soc. 87 (2013), no. 1, 289--303.

\bibitem{Uspenskij1990} V. V. Uspenskij, {\it On the subgroup of isometries of the Urysohn universal metric space}, Comment. Math. Univ. Carolin. 31 (1990), 181--182.

\bibitem{YVH_Hilbert} V.H Ya\~nez, {\em Generalizing minimal almost periodicity via local compactness, NSS and Lie properties}, Rev. Real Acad. Cienc. Exactas Fis. Nat. Ser. A-Mat (2026) 120:2, pp: 1--25.

\end{thebibliography}
\end{document}